\documentclass[11pt]{article}
\usepackage{amsmath,amsthm,url}
\usepackage{authblk,xcolor}
\usepackage[top=1in, bottom=1in, left=1.25in, right=1.25in]{geometry}

\newtheorem{theorem}{Theorem}[section]
\newtheorem{example}{Example}[section]

\newtheorem{lemma}[theorem]{Lemma}
\newtheorem{remark}[theorem]{Remark}

\newcommand{\MS}[1]{\mathsf{#1}}
\newcommand{\BB}{\ensuremath{\mathcal{B}}}
\newcommand{\YY}{\ensuremath{\mathcal{Y}}}

\title{An introduction to local differential privacy protocols using block designs}
\author[1]{Maura B.\ Paterson}
\affil[1]{School of Computing and Mathematical Sciences, Birkbeck, University of London, Malet St, London WC1E 7HX, UK}
\author[2]{Douglas R.\ Stinson\thanks{D.R. Stinson’s research was supported by an NSERC General Research Funds grant (University of Waterloo).}}
\affil[2]{David R.\ Cheriton School of Computer Science\\University of Waterloo\\ Waterloo ON, N2L 3G1, Canada}

\begin{document}
\maketitle

\begin{abstract}
The design of protocols for \emph{local differential privacy} (or \emph{LDP}) has been a topic of considerable research interest in recent years. LDP protocols utilise the randomised encoding of outcomes of an experiment using a transition probability matrix (TPM).
Several authors have observed that balanced incomplete block designs (BIBDs) provide nice examples of TPMs for  LDP protocols. Indeed, it has been shown that such BIBD-based LDP protocols provide optimal estimators. 

In this primarily expository paper, we give a detailed introduction to LDP protocols and their connections with block designs. We prove that a subclass of LDP protocols known as pure LDP protocols are equivalent to $(r,\lambda)$-designs  (which contain balanced incomplete block designs as a special case). An unbiased estimator for an LDP scheme is a left inverse of the transition probability matrix. We show that the optimal estimators for  BIBD-based TPMs are precisely those obtained from the Moore-Penrose inverse of the corresponding TPM. We also review some existing work on optimal LDP protocols in the context of pure protocols.
\end{abstract}

\section{Introduction}
\label{intro.sec}

In 1965, Warner \cite{War} proposed the ``randomised response'' protocol. The setting is a survey where participants are required to provide a yes-no answer regarding a certain attribute. After collecting a number of responses, the goal is to determine the probability $p$ that the attribute exists in a random participant. The issue is that participants may be reluctant to give correct answers due to privacy concerns. 

Warner proposed that each participant should provide the correct response with some specified probability $\theta > 1/2$. With probability $1 - \theta$ the participant will give the incorrect response. 
Given a number $t$ of responses, it is fairly straightforward to determine an unbiased estimator $\hat{p}$ for $p$
(an estimator $\hat{p}$ is \emph{unbiased} if $E[\hat{p}] = p$). The accuracy of the estimator is measured by computing the variance $E[(p - \hat{p})^2]$. Intuitively, as $\theta$ increases, the accuracy of the survey increases but privacy decreases. 

In a \emph{local differential privacy} (or \emph{LDP}) protocol, 
we have $n$ possible outcomes of each trial of some experiment.
We define a perturbation function (or \emph{local randomiser}) $f : X \rightarrow Y$.
 $X$ 
 is a set of $n$ possible outcomes. 
 $Y$ 
 is the set of $m$ possible ``perturbed'' or randomised values reported by participants. To avoid trivial cases, we will require that $m\geq 2$ and $n \geq 2$. 

There is also a parameter $\theta$, where $0 < \theta < 1$.
For each $x$,
there is a specified subset $Y_x \subseteq Y$, where $Y_x \neq \emptyset$ or $Y$.
The perturbation function $f$ maps $x$ to an element of $Y_x$ with probability $\theta$, and
$x$ is mapped to an element of $Y \setminus Y_x$ with probability $1 -\theta$.
More precisely, we have
\begin{equation}
\label{basic.eq}
\mathsf{Prob}[Y = y | X = x ] = \mathsf{Prob}[f(x) = y] = 
\begin{cases}
\frac{\theta}{|Y_x|} & \text{if $y \in Y_x$}\\
\frac{1-\theta}{|Y \setminus Y_x|} & \text{if $y \in Y \setminus Y_x$}.
\end{cases}
\end{equation}
This is interpreted as follows. Given the outcome $x$, a biased coin is flipped. With probability $\theta$, choose an element  $y \in Y_x$ uniformly at random, and with probability $1- \theta$, choose an element  $y \in Y \setminus Y_x$ uniformly at random. The resulting value $y$ is then reported. 

\begin{example}
Suppose $X = Y = \{0,1\}$, $Y_0 = \{0\}$ and $Y_1 = \{1\}$. Then we obtain the ``randomised response'' scheme of Warner.
\end{example}

For later reference, we define 
\begin{equation}
\label{TP.eq}
\alpha_{x,1} = \frac{\theta}{|Y_x|} \quad \text{and} \quad
\alpha_{x,2} = \frac{1-\theta}{|Y \setminus Y_x|}.
\end{equation}

In general, $\alpha_{x,1}$ and $\alpha_{x,2}$ will depend on $x$. However, we will later be investigating schemes where $\alpha_{x,1}$ and $\alpha_{x,2}$ are constants, independent of $x$.

We think of the points in $Y_x$ as being ``high probability'' perturbations of $x$ and the remaining points as being 
``low probability'' perturbations. 
So $\theta$ should be chosen such that $\alpha_{x,1}$ is somewhat greater than
$\alpha_{x,2}$, for all $x$.

The above ideas are quantified by the definition of \emph{differential privacy} from Duchi, Jordan and Wainwright \cite{DJW}.
Let $\epsilon > 0$ be a specified privacy parameter. Then we say that the local randomiser
 $f : X \rightarrow Y$ is \emph{$\epsilon$-locally differentially private} (or \emph{$\epsilon$-LDP}) provided that
\begin{equation}
\label{LDP.eq}
\frac{\mathsf{Prob}[Y = y | X = x ]}{\mathsf{Prob}[Y = y | X = x']} 
 \leq e^{\epsilon}
\end{equation}
for all $y \in Y$ and all $x, x' \in X$, $x \neq x'$.
The quantity $e^{\epsilon}$ is the \emph{differential privacy ratio}.

In the case where $|Y_x|$ is constant (independent of $x$), we can use (\ref{TP.eq}) to simplify (\ref{LDP.eq}) as follows:
\begin{equation}
\label{LDP2.eq}
\alpha_{x,1} \leq e^{\epsilon} \alpha_{x,2}
\end{equation}
should hold for all $x$.

\medskip

We now describe LDP mechanisms derived from appropriate set systems. 
Suppose we have $m$ points in $Y$ and $n$ sets (or blocks) $Y_x$ ($x \in X$), which are subsets of $Y$. 
We denote $\mathcal{Y} = \{ Y_x: x \in X \}$ and we investigate properties of the set system
 $(Y, \mathcal{Y})$.
 
We can also consider the dual set system consisting of the $n$ points in $X$ and the $m$ blocks
$B_y$ ($y \in Y$) defined as  $B_y = \{ x : y \in Y_x\}$.\footnote{The dual set system just interchanges the roles of points and blocks. Equivalently, we transpose the incidence matrix of the set system.} 
Observe that $x \in B_y$ if and only if $y \in Y_x$. 
We denote $\mathcal{B} = \{ B_y: y \in Y\}$.
Then  $(X,\mathcal{B})$ is a set system on $n$ points and $m$ blocks.


\begin{example}
\label{E1.exam}
Suppose $n=4$, $m=6$, $X = \{1,2,3,4\}$ and $Y = \{\MS{a},\MS{b},\MS{c},\MS{d},\MS{e},\MS{f}\}$, and let
the blocks $Y_1, \dots , Y_4$ in $\mathcal{Y}$ be as follows:
\[
\begin{array}{llll}
Y_1 = \{ \MS{a},\MS{b},\MS{c}\} &
Y_2 = \{ \MS{a},\MS{d},\MS{e}\}&
Y_3 = \{ \MS{b},\MS{d},\MS{f}\}&
Y_4 = \{ \MS{c},\MS{e},\MS{f}\}.
\end{array}
\]
The set system $(Y, \mathcal{Y})$ has six points and four blocks.

\smallskip

The dual set system $(X,\mathcal{B})$ has four points and six blocks:
\[
\begin{array}{llllll}
B_{\MS{a}} = \{ 1,2\}&
B_{\MS{b}} = \{ 1,3\}&
B_{\MS{c}} = \{ 1,4\}&
B_{\MS{d}} = \{ 2,3\}&
B_{\MS{e}} = \{ 2,4\}&
B_{\MS{f}} = \{ 3,4\}.
\end{array}
\]
$(X,\mathcal{B})$ is a $(4,6,3,2,1)$-BIBD.

\medskip

The following table shows the resulting local randomiser. The $(y,x)$-entry in the table is the
probability that $x \in X$ is mapped to $y \in Y$, i.e., $\mathsf{Prob}[f(x) = y]$.

\[
\begin{array}{c|cccc}
 & 1 & 2 & 3 & 4  \\ \hline
 \MS{a} & \frac{\theta}{3} & \frac{\theta}{3} & \frac{1-\theta}{3} & \frac{1-\theta}{3} \rule[2mm]{0cm}{3mm} \\
 \MS{b} & \frac{\theta}{3} & \frac{1-\theta}{3} & \frac{\theta}{3} & \frac{1-\theta}{3} \rule[2mm]{0cm}{3mm} \\
 \MS{c} & \frac{\theta}{3} & \frac{1-\theta}{3} & \frac{1-\theta}{3} & \frac{\theta}{3} \rule[2mm]{0cm}{3mm} \\
 \MS{d} & \frac{1-\theta}{3} & \frac{\theta}{3} & \frac{\theta}{3} & \frac{1-\theta}{3} \rule[2mm]{0cm}{3mm} \\
 \MS{e} & \frac{1-\theta}{3} & \frac{\theta}{3} & \frac{1-\theta}{3} & \frac{\theta}{3} \rule[2mm]{0cm}{3mm} \\
 \MS{f} & \frac{1-\theta}{3} & \frac{1-\theta}{3} & \frac{\theta}{3} & \frac{\theta}{3} \rule[2mm]{0cm}{3mm} \\
\end{array}
\]
\end{example}

\medskip

\subsection{Our Contributions}

This paper is mainly expository in nature, although we do present a few new results.
We should emphasize that our main goal is to treat LDP protocols in a uniform combinatorial setting.

In Section \ref{BDRR.sec}, we discuss LDP protocols based on balanced incomplete block designs (BIBDs); these protocols have been termed ``block design randomised response'' in the literature.
Section \ref{pure.sec} concerns so-called \emph{pure} LDP protocols. We prove that a  pure LDP protocol is equivalent to an $(r,\lambda)$- design; these designs  include BIBDs as special cases.
In Section \ref{est.sec}, we turn our attention to estimators. We derive formulas to quantify the effectiveness of unbiased estimators for the (unknown) underlying probability distribution. An unbiased estimator for an LDP scheme is a left inverse of the transition probability matrix.  In Section \ref{variance.sec}, we discuss how to compute the variance of an unbiased estimator (this is sometimes called ``risk''). Explicit formulas are presented in the case of  pure protocols.
In Section, \ref{optimal.sec}, we present some results on 
LDP protocols based on BIBDs; they are in fact optimal in a certain sense. One new result  we prove is that the optimal estimators for BIBD-based LDP protocols are precisely those obtained from the Moore-Penrose inverse of the corresponding TPM.
We also review some previous work on optimality of TPMs in a general setting.

We finish this section by presenting in Table \ref{tab1} the notation that is used in the rest of this paper.

\begin{table}[tb]
\caption{Notation used in this paper}
\label{tab1}
\begin{center}\begin{tabular}{|l|l|l|}
\hline
$X = \{x_1, \dots , x_n\}$  & set of  possible outcomes & Section \ref{intro.sec}\\
$Y = \{y_1, \dots , y_m\}$ & set of possible perturbed values & Section \ref{intro.sec}\\
$Y_x$ & high probability outcomes for a given $x \in X$ & Section \ref{intro.sec}\\
$\theta$ & probability of a high probability outcome & Section \ref{intro.sec}\\
$e^{\epsilon}$ & differential privacy ratio & Section \ref{intro.sec}\\ \hline
$(v,b,r,k,\lambda)$ & BIBD parameters ($n = v$, $m = b$) & Section \ref{BDRR.sec}\\
$(X, \BB )$ & points and blocks of BIBD & Section \ref{BDRR.sec}\\
$(Y, \YY )$ & points and blocks of dual BIBD & Section \ref{BDRR.sec}\\
$\alpha_1$ & probability of each possible high probability outcome & Section \ref{BDRR.sec}\\
$\alpha_2$ & probability of each possible low probability outcome & Section \ref{BDRR.sec}\\
$A$ & $b$ by $v$ incidence matrix of a BIBD & Section \ref{BDRR.sec}\\
$J$ & all-$1$'s matrix & Section \ref{BDRR.sec}\\
$Q$ & $b$ by $v$ transition probability matrix (TPM) & Section \ref{BDRR.sec}\\ \hline
$p^*$ & true positive rate of a pure LDP protocol & Section \ref{pure.sec}\\
$q^*$ & false positive rate of a pure LDP protocol & Section \ref{pure.sec}\\ \hline
$\pi = (p_1, \dots, p_n)$ & probability distribution on $X$ & Section \ref{est.sec}\\
$t$ & number of samples &  Section \ref{est.sec}\\
$T_j$ & number of occurrences of points in $Y_j$ &  Section \ref{est.sec}\\
$\tilde{p}_j$ & estimator for $p_j$ &  Section \ref{est.sec}\\
$f_1, \dots , f_m$ & frequencies of perturbed samples &  Section \ref{est.sec}\\
$L$ & unbiased linear estimator &  Section \ref{est.sec}\\
$Q^+$ & Moore-Penrose left inverse of $Q$ &  Section \ref{est.sec}\\
$B(t,p)$ & binomial distribution with $t$ trials and probability $p$ & Section \ref{est.sec}\\
$\mathrm{Var}(\; )$ & variance of a random variable & Section \ref{est.sec}\\
$E[ \; ]$ & expectation of a random variable & Section \ref{est.sec}\\
$\rho = (\rho_1, \dots , \rho_m)$ & probability distribution on $Y$ induced by $\pi$ and $Q$ 
& Section \ref{est.sec}\\ \hline
$D_{\rho}$  & diagonal matrix with 
$\rho$ on the main diagonal  
& Section \ref{variance.sec}\\
$\mathbf{R}(Q,L;\pi)$ & risk function & Section \ref{variance.sec}\\
$V$ & covariance matrix & Section \ref{variance.sec}\\
\hline
\end{tabular}
\end{center}
\end{table}

\section{Block Design Randomized Response}
\label{BDRR.sec}

Many authors have investigated the construction of LDP protocols using various kinds of balanced incomplete block designs (BIBDs), including Hadamard designs and projective geometries, to name two examples. 
Relevant references include \cite{ASZ,FNNT,Gen,PNL,YB}. The paper by Park, Nam and Lee provided a unified treatment of these approaches. We briefly review  the description of this approach, which is termed
``Block Design Randomized Response''  in  \cite[Definition 5]{Gen}.

First, we recall the definition of a BIBD. A   $(v,b,r,k ,\lambda)$-BIBD is a set system $(X,\BB)$ that satisfies the following properties:
\begin{enumerate}
\item $X$ is a set of $v$ \emph{points}
\item $\BB$ is a set of $b$ subsets of $X$ called \emph{blocks}
\item $|B| = k$ for every $B \in \BB$
\item every point occurs in exactly $r$ blocks
\item every pair of points is contained in exactly $\lambda$ blocks.
\end{enumerate}
Sometimes the parameters of a BIBD are written as a triple $(v,k,\lambda)$. The other two parameters, $r$ and $b$, are determined from the following fundamental equations:
\[ bk = vr\] and
\[ \lambda (v-1) = r(k-1).\]

It is instructive to refer to the construction given in 
Example \ref{E1.exam}.
As in this example, $(X,\mathcal{B})$ will be a $(v,b,r,k ,\lambda)$-BIBD and 
 $(Y, \mathcal{Y})$ is the dual design. Hence $n = |X| = v$ and $m = |Y| = b$. Each block of $(Y, \mathcal{Y})$ has cardinality $r$ and any two distinct blocks in $\mathcal{Y}$ intersect in exactly $\lambda$ points.
 
 Using the facts that $|Y_x| = r$ for all $x$ and $|Y| = b$, we have 
\begin{align}
 \label{alpha_1.eq} \alpha_{x,1} &= \frac{\theta}{r} \quad \text{and} \\
\label{alpha_2.eq}\alpha_{x,2} &= \frac{1- \theta}{b  -r}
\end{align}
for all $x$. Since $\alpha_{x,1}$ and
$\alpha_{x,2}$ are constants, we denote
$\alpha_1 = \alpha_{x,1}$ and
$\alpha_2 = \alpha_{x,2}$.

As in (\ref{LDP2.eq}), the differential privacy ratio  $\alpha_1/\alpha_2$ 
should be at most $e^{\epsilon}$. So we require that 
\begin{equation}
\label{ep2.eq}
 e^{\epsilon} \geq \frac{\frac{\theta}{r}}{\frac{1-\theta}{b-r}} = 
\left( \frac{\theta}{1 - \theta} \right) \left( \frac{b-r}{r} \right).
\end{equation}
For convenience, assume  we have equality in (\ref{ep2.eq}). Then we have 
\begin{align}
\nonumber \theta &= \frac{r e^{\epsilon}}{b  + r (e^{\epsilon} - 1)}\\
\label{p_1.eq}\alpha_1 &= \frac{e^{\epsilon}}{b  + r (e^{\epsilon} - 1)} \quad \text{and} \\
\label{p_2.eq}\alpha_2 &= \frac{1}{b  + r (e^{\epsilon} - 1)}.
\end{align}

\medskip

\begin{remark} The $k$-subset mechanism presented in \cite{Wang2019} is an example of Block Design Randomized Response. It uses the (trivial) $\left( v,\binom{v}{k}, \binom{v-1}{k-1}, k, \binom{v-1}{k-2}\right)$-BIBD consisting of all
$k$-subsets of a set of $v$ points. The formulas in  \cite[equation (3)]{Wang2019} are the same as our equations (\ref{alpha_1.eq}) and (\ref{alpha_2.eq}). The paper \cite{YB} uses a similar approach.
\end{remark}

\bigskip

The \emph{transition probability matrix} (or \emph{TPM}) of an LDP protocol is just the table presented in Example \ref{E1.exam}. In general, the TPM is   a $b$ by $v$ matrix $Q$ where the rows are indexed by the blocks in $\BB$ and the columns are indexed by the points in $X$. For all $B \in \BB$ and all $x \in X$, the entry in row $B$ and column $x$ of $Q$ 
is $\mathsf{Prob}[f(x) = B]$.

Suppose that $A$ is the $b$ by $v$ incidence matrix of the $(v,b,r,k, \lambda)$-BIBD $(X,\mathcal{B})$. Since $(X,\mathcal{B})$ is a BIBD, the following equation holds:
\begin{equation}
\label{BIBD.eq}
A^T A = \lambda J_{v \times v} + (r - \lambda)I_{v \times v},
\end{equation}
where $J_{v \times v}$ is a $v$ by $v$ matrix of $1$'s.

Then the TPM $Q$ can be written as 
\begin{equation}
\label{TPM.eq} Q = \alpha_1 A + \alpha_2(J_{b \times v} - A) = A (\alpha_1 - \alpha_2) + \alpha_2 J_{b \times v},
\end{equation}
where $J_{b \times v}$ is a $b$ by $v$ matrix of $1$'s.
When $\alpha_1$ and $\alpha_2$ have the values given in (\ref{p_1.eq}) and (\ref{p_2.eq}),  we have
\[ Q = \alpha_2 (A (e^{\epsilon} - 1) + J_{b \times v}).\]

\section{Pure LDP Protocols}
\label{pure.sec}

In this section, we study a useful subclass of LDP protocols (namely, pure protocols) and show that they are equivalent to a relaxation of BIBDs known as $(r,\lambda)$-designs.
First, we recall the definition of a \emph{pure LDP protocol} from \cite[Definition 3]{WBLJ}.
In a pure LDP protocol, there should exist probabilities $p^* > q^*$ such that 
\begin{align*}
\mathsf{Prob}[f(x) \in \{y : x \in B_y \}]  &= p^* \quad \text{for all $x$}\\
\mathsf{Prob}[f(x')\in \{y  : x \in B_y \}]  &= q^* \quad \text{for all $x' \neq x$.}
\end{align*}
These two conditions are equivalent to the following:
\begin{align*}
\mathsf{Prob}[f(x) \in Y_x]  &= p^* \\
\mathsf{Prob}[f(x')\in Y_x]  &= q^* \quad \text{for all $x' \neq x$.}
\end{align*}
It may be helpful to think of $p^*$ as being a \emph{true positive rate} and $q^*$ as being a \emph{false positive rate} for the protocol.

Note that, by definition, we have the following formula:
\begin{equation}
\label{p*.eq}
p^* = \theta.
\end{equation}
The second probability, $q^*$, is computed as follows:

\begin{equation}
\label{q*.eq}  q^* 
= |Y_x \cap Y_{x'}| \left(\frac{\theta}{|Y_{x'}|}\right) 
+ |Y_x \setminus Y_{x'}| \left(\frac{1-\theta}{|Y \setminus Y_{x'}|}\right)
\end{equation}
for all $x \neq x'$.

\bigskip

We provide a detailed derivation of (\ref{q*.eq}).
We are assuming $x' \neq x$ and we are interested in the situation where $f(x') = y \in Y_x$. There are two cases to consider. 
\begin{enumerate}
\item If $y \in Y_x \cap Y_{x'}$, then, from (\ref{basic.eq}), we have $ \mathsf{Prob}[f(x') = y] = \theta / |Y_x'|$.
\item If $y \in Y_x \setminus Y_{x'}$, then from (\ref{basic.eq}), we have $ \mathsf{Prob}[f(x') = y] = (1-\theta) / |Y \setminus Y_{x'}|$.
\end{enumerate}
Equation (\ref{q*.eq}) then follows easily.

\bigskip

We illustrate by returning to Example \ref{E1.exam}. 
\begin{example}
Suppose we want to compute (\ref{q*.eq}) when $x = 1$ and $x'=2$.
That is, we want to determine $\mathsf{Prob}[f(2)\in Y_1]$.
We have the following probabilities:
\begin{align*}
\mathsf{Prob}[f(2) = \MS{a}]  &= \frac{\theta}{3}\\
\mathsf{Prob}[f(2) = \MS{b}]  &= \frac{1-\theta}{3}\\
\mathsf{Prob}[f(2) = \MS{c}]  &= \frac{1-\theta}{3}.
\end{align*}
We have $Y_1 = \{ \MS{a},\MS{b},\MS{c}\}$. Hence, 
\[ q^* = \mathsf{Prob}[f(2) \in Y_1] 
= \frac{\theta}{3} + 2\left(\frac{1-\theta}{3}\right) = \frac{2 - \theta}{3}.\]
\end{example}

\medskip

We now investigate conditions (in a pure LDP protocol) under which $q^*$ is independent of the choices of $x$ and $x'$. We will require that this property holds for all possible values of $\theta$. Denote $|X| = n$ and $|Y| = m$.
For convenience, define $\ell_x = |Y_x|$ for all $x$ and define
$\mu_{x,x'} = |Y_x \cap Y_{x'}|$ for all $x \neq x'$.  Note that $1 \leq \ell_x \leq m-1$ for all $x$.

The value
\begin{equation}
\label{q.eq} q^* 
= \mu_{x,x'} \left(\frac{\theta}{\ell_{x'}}\right) 
+ (\ell_x - \mu_{x,x'}) \left(\frac{1-\theta}{m - \ell_{x'}}\right)
\end{equation}
should be independent of the  choices of $x$, $x'$ (where $x \neq x'$) and $\theta$.
We can rewrite (\ref{q.eq}) as follows:
\begin{equation}
\label{q2.eq}
 q^* 
 = \theta \left( \frac{\mu_{x,x'}}{\ell_{x'}}   - \frac{\ell_x - \mu_{x,x'}}{m - \ell_{x'}} \right)
+ \frac{\ell_x - \mu_{x,x'}}{m - \ell_{x'}}\end{equation}
for all $x \neq x'$.

Recall that we require that (\ref{q2.eq}) holds for all values of $\theta$.
Since $q^*$ is a linear function of $\theta$, we require that 
the two ``coefficients''  
\[  \frac{\mu_{x,x'}}{\ell_{x'}}   - \frac{\ell_x - \mu_{x,x'}}{m - \ell_{x'}} \quad \text{and} \quad \frac{\ell_x - \mu_{x,x'}}{m - \ell_{x'}}\]
are independent of $x$ and $x'$. That is, there are constants $c_1$ and $c_2$ such that
\begin{equation}
\label{eq1}
c_1 = \frac{\ell_x - \mu_{x,x'}}{m - \ell_{x'}}
\end{equation}
and 
\begin{equation}
\label{eq2}
c_2 = \frac{\mu_{x,x'}}{\ell_{x'}}   - \frac{\ell_x - \mu_{x,x'}}{m - \ell_{x'}}
\end{equation}
for all $x \neq x'$.

It follows immediately from (\ref{eq1}) and (\ref{eq2}) that 
\begin{equation}
\label{c1c2.eq} c_1+c_2 = \frac{\mu_{x,x'}}{\ell_{x'}}
\end{equation}
for all $x \neq x'$. 

\begin{lemma}
If $\mu_{x,x'}= 0$ for some $x \neq x'$, then $\mu_{x,x'}= 0$ for all $x \neq x'$.
\end{lemma}
\begin{proof}
Since $c_1 + c_2 = \mu_{x,x'}/\ell_{x'}$ is a constant for all $x \neq x'$, this is an immediate  consequence of  (\ref{c1c2.eq}). 
\end{proof}

There are thus two cases to consider:
\begin{description}
\item[case 1] $\mu_{x,x'} \neq 0$ for all $x \neq x'$.
\item[case 2] $\mu_{x,x'} = 0$ for all $x \neq x'$.
\end{description}

First, we consider case 1. 
Interchanging the roles of $x$ and $x'$ in (\ref{c1c2.eq}), we have, for all $x \neq x'$, that
\begin{equation}
\label{ellxx'.eq}
 \frac{\mu_{x,x'}}{\ell_x} = \frac{\mu_{x,x'}}{\ell_{x'}}.
\end{equation}

Since  $\mu_{x,x'} \neq 0$, it follows that $\ell_x = \ell$ for all $x$, where $\ell \neq 0$ is a constant.
Substituting $\ell = \ell_x = \ell_{x'}$ into (\ref{eq1}), we see that
\[ c_1 = \frac{\ell - \mu_{x,x'}}{m - \ell}\] for all $x \neq x'$.
So all the values $\mu_{x,x'}$ are constant, i.e., $\mu = \mu_{x,x'}$ for all $x \neq x'$, where $\mu$ is a constant.

\medskip

So far, we have proven that the set system 
consisting of the $m$ points in $Y$ and the $n$ blocks $Y_x$ ($x \in X$)
satisfies the following two properties:
\begin{enumerate}
\item every block has size $\ell$
\item the intersection of any two distinct blocks has cardinality $\mu$.
\end{enumerate}
Therefore, 
we have the following restatement of (\ref{q*.eq}) when the set system $(Y,\mathcal{Y})$ satisfies the two properties above.
\begin{equation}
\label{q*2.eq}
q^* = \frac{\theta \mu}{\ell} 
+  \frac{(1-\theta)(\ell - \mu)}{m - \ell}.
\end{equation}

\medskip

Now let's consider the dual set system $(X,\mathcal{B})$. %
Since $(X,\mathcal{B})$ is the dual set system to $(Y,\mathcal{Y})$,  it satisfies the following properties
\begin{enumerate}
\item every point occurs in exactly $\ell$ blocks
\item every pair of distinct points is contained in exactly $\mu$ blocks.
\end{enumerate}
This type of combinatorial design is often termed an $(r,\lambda)$-design (see \cite[\S VI.49]{CD})
where, in this case, $r = \ell$ and $\lambda = \mu$. 
Note that an $(r,\lambda)$-design is a 
BIBD if every block contains the same number of points. 
We note that $(r,\lambda)$-designs were suggested for use as pure LDP schemes in \cite{PNL}, where they were termed
``regular and pairwise-balanced designs.''

\medskip

We summarize our discussion of case 1 as follows.
\begin{lemma}
\label{case1.lem} Suppose $\mu_{x,x'} \neq 0$ for all $x \neq x'$. Then the set system $(X,\mathcal{B})$, as defined above,  is an $(r,\lambda)$-design.
\end{lemma}

\medskip
Let's look at a small $(r,\lambda)$-design that is not a BIBD. 

\begin{example}
\label{(3,1).exam}
We delete a point from a 
$(7,3,1)$-BIBD, creating the following $(3,1)$-design on $m=6$  points and $n=7$ blocks.
\begin{align*}
X &= \{1, \dots , 6\}\\
\mathcal{B} &= \{ \{1,2,4\}, \{2,3,5\}, \{3,4,6\}, \{4,5\}, 
\{5,6,1\}, \{6,2\}, \{1,3\} \}.
\end{align*}
Let 
\begin{align*}
Y &= \{\MS{a},\MS{b},\MS{c},\MS{d},\MS{e},\MS{f},\MS{g}\}\\
\mathcal{Y} &= \{ \{\MS{a},\MS{e},\MS{g}\}, \{\MS{a},\MS{b},\MS{f}\}, \{\MS{b},\MS{c},\MS{g}\}, 
\{\MS{a},\MS{c},\MS{d}\}, 
\{\MS{b},\MS{d},\MS{e}\}, \{\MS{c},\MS{e},\MS{f}\}\}.
\end{align*}
Then $(X, \mathcal{B})$ and $(Y, \mathcal{Y})$ are dual designs.
It can be seen that all blocks in $\mathcal{Y}$ have cardinality three and any two distinct blocks intersect in exactly one point.

For the LDP $(Y, \mathcal{Y})$, the formula (\ref{q*2.eq}) yields
\[ q^* = \frac{\theta }{3} 
+  \frac{2(1-\theta)}{3} = \frac{2-\theta}{3}.
\] 
\end{example}

\bigskip

We now consider case 2, where we assume that $\mu_{x,x'} = 0$ for all $x \neq x'$.

 \begin{lemma}
\label{sum.lem}
Suppose $\mu_{x,x'} = 0$ for all $x \neq x'$. Then, for all $x \neq x'$, $\ell_x = \ell_{x'} $ or $\ell_x + \ell_{x'} = m$.
\end{lemma}
 \begin{proof}
Interchanging the roles of $x$ and $x'$ in (\ref{eq1}),
 we have the following for any $x \neq x'$:
\[ \frac{\ell_x - \mu_{x,x'}}{m - \ell_{x'}} = \frac{\ell_{x'} - \mu_{x,x'}}{m - \ell_{x}} .\]
Since $\mu_{x,x'} = 0$, we obtain
\[ \ell_x(m - \ell_{x}) = \ell_{x'}(m - \ell_{x'}),\]
which simplifies to give
\[ m (\ell_x - \ell_{x'})= \ell_x^2 - \ell_{x'}^2.\]
If $\ell_x \neq \ell_{x'}$, we have
$ m  = \ell_x + \ell_{x'}.$
\end{proof}

\begin{lemma} 
\label{case2.lem}
Suppose $\mu_{x,x'} = 0$ for all $x \neq x'$. Then (at least) one of the following holds.
\begin{enumerate}
\item $|X| = 2$.
\item $\ell_x$ is a constant.
\end{enumerate}
\end{lemma}
\begin{proof}
Suppose $|X| \geq 3$ and not all the values $\ell_x$ are identical. From Lemma \ref{sum.lem},  the $\ell_x$'s take on precisely two values. There must exist three distinct elements
$x, x', x'' \in X$ such that $\ell_x = \ell_{x'} \neq \ell_{x''}$.
From (\ref{eq1}), we have
\[ c_1 = \frac{\ell_x }{m - \ell_{x'}} = \frac{\ell_x }{m - \ell_{x''}}.\]
This implies that $\ell_{x'} = \ell_{x''}$, which is a contradiction.
We conclude that $|X| = 2$.
\end{proof}

\bigskip

In part 1 of Lemma \ref{case2.lem}, we have $|X| = 2$. Hence, there are precisely two disjoint blocks in $\mathcal{Y}$. If the two blocks are of different sizes, say $\ell_1 \neq \ell_2$, then $\ell_1 + \ell_2 = m$. In part 2, $\mathcal{Y}$ is a set of disjoint blocks, which are all necessarily of the same size, say $\ell$. In this case, the dual of $(Y,\mathcal{Y})$ is an $(r,\lambda)$-design with $r = \ell$ and $\lambda = 0$.

\begin{example}
\label{unequal.ex}
Suppose $X = \{0,1\}$, $Y = \{a,b,c\}$, $Y_0 = \{ a \}$ and $Y_1 = \{b,c\}$. This yields a pure LDP scheme where
$|Y_0| \neq |Y_1|$.\footnote{This example was provided by the referee.} 
\end{example}

We  summarize the two cases as follows. 

\begin{theorem} 
\label{main.thm}
Suppose $|X| \geq 3$.
The following are equivalent:
\begin{enumerate}
\item The perturbation function $f: X \rightarrow Y$ 
based on the set system $(Y,\mathcal{Y})$ is a pure LDP scheme.
\item $(Y,\mathcal{Y})$ has constant block size $\ell$ and every pair of blocks intersects in exactly $\mu$ points.
\item The dual design $(X,\mathcal{B})$ is an $(r,\lambda)$-design, where $r = \ell$ and $\lambda = \mu$.
\end{enumerate}
\end{theorem}

Observe that $\alpha_1 = \alpha_{x,1}$ and $\alpha_2 = \alpha_{x,2}$ are constants (independent of $x$) in a pure LDP scheme.
The values of the parameters of an LDP protocol based on an $(r,\lambda)$-design are summarized as follows.

\begin{theorem}
\label{BIBD.thm}
An LDP protocol based on an $(r,\lambda)$-design (which includes $(v,b,r,k,\lambda)$-BIBDs as special cases) has the following parameters:
\[
\begin{array}{l}
|X| = v \quad \text{and} \quad  |Y| = b\vspace{.1in}\\
\theta = \frac{r e^{\epsilon}}{b  + r (e^{\epsilon} - 1)}\vspace{.1in}\\
\alpha_1 = \frac{\theta}{r} = \frac{e^{\epsilon}}{b  + r (e^{\epsilon} - 1)}
\quad \text{and} \quad
\alpha_2 = \frac{1- \theta}{b  -r} = \frac{1}{b  + r (e^{\epsilon} - 1)}\vspace{.1in}\\
\label{Pstarqstar.eq}
p^* = \theta \quad \text{and} \quad q^* = \frac{\theta \lambda}{r} 
+  \frac{(1-\theta)(r - \lambda)}{b-r}.
\end{array} 
\]
\end{theorem}

\medskip

From the equations listed in Theorem \ref{BIBD.thm}, the following can be verified easily.
\[
\begin{array}{l} 
p^* = \alpha_1 r\vspace{.1in}\\
q^* = \alpha_1 \lambda + \alpha_2 (r - \lambda)\vspace{.1in}\\
p^* - q^* = (r - \lambda)(\alpha_1 - \alpha_2).
\end{array}
\]
Using the above equations, we can rewrite (\ref{TPM.eq}) as
\begin{equation}
\label{TPM2.eq}
Q = \left( \frac{p^*-q^*}{r-\lambda}\right)  A + \left( \frac{p^*}{r} - \frac{p^*-q^*}{r-\lambda} \right) J_{b \times v}.
\end{equation}

\section{Estimators}
\label{est.sec}

Suppose we denote $X = \{x_1, \dots , x_n\}$ and $Y = \{y_1, \dots , y_m\}$. We are given a set of randomised values which arise from an unknown probability distribution $\pi$ on 
$X$ after being modified by a TPM, say $Q$. We want to estimate this probability distribution
$\pi = (p_1, \dots, p_n)$, where \[p_j = \mathsf{Prob}[X = x_j],\] for $1 \leq j \leq n$.

Associated with the TPM $Q$, we have the two dual designs $(X, \mathcal{B})$ and  $(Y, \mathcal{Y})$.
For now, suppose $(X, \mathcal{B})$ is an $(r,\lambda)$-design (or a $(v,b,r,k,\lambda)$-BIBD) with $n = v$ and $m = b$.
We denote the blocks in $\mathcal{B}$ by $B_1, \dots , B_m$.
The blocks in $(Y, \mathcal{Y})$ are denoted $Y_1, \dots , Y_n$. 
Each block $Y_i$ has cardinality $r$.

Suppose that we have $t$ perturbed values (samples) and we want to estimate $p_j$.
We can tally the number of occurrences of points in $Y_j$ among the $t$ samples. Call this quantity
$T_j$.
Then (as discussed in \cite{WBLJ}, for example), an unbiased estimator for $p_j$ is given by the formula
\begin{equation}
\label{estimator}
\tilde{p}_j = \frac{T_j -tq^*}{t(p^* - q^*)}.
\end{equation}

It may be helpful to provide an intuitive explanation of this formula.
If the matrix $Q$ was just the identity matrix, then $T_i$ would be the number of occurrences of $x_i$, and the sample mean $T_i/t$ would be an unbiased estimator for the population mean $p_i$.
 
When we apply a TPM for a pure protocol, then a single trial contributes $+1$ to $T_i$ with probability 
$p^*  p_i+q^*(1-p_i)=(p^*-q^*)p_i+ q^*$; 
otherwise, it contributes $0$.  So we can view this quantity as the new population mean, and $T_i/t$ is an unbiased estimator for it; i.e., 
\begin{equation}
\label{est.eq}
E\left[\frac{T_i}{t}\right]= (p^*-q^*)p_i+q^*.
\end{equation} 
We want an unbiased estimator for $p_i$, so we can use linearity of expectation. 
Hence, it follows from (\ref{est.eq}) that
\[E \left[ \frac{\frac{T_i}{t}-q^*}{p^*-q^*} \right] =p_i,\]
which is equivalent to (\ref{estimator}).
We’re in effect subtracting the expected contribution arising from ``incorrect'' samples that aren’t $x_i$, and then scaling the remaining contribution from the ``correct'' samples $x_i$ to account for the fact that they 
don't always contribute $+1$ to $T_i$.

\medskip

Suppose we denote the frequencies of the (perturbed) samples by
$f_1, \dots , f_m$, where $f_i$ denotes the number of occurrences of $y_i$, for $1 \leq i \leq m$.
Then 
\[ T_j = \sum_{y_i \in Y_j} f_i.\]
We can rewrite (\ref{estimator}) as follows:
\begin{equation}
\label{estimator2}
\tilde{p}_j  = \frac{\sum_{y_i \in Y_j} f_i -tq^*}{t(p^* - q^*)}.
\end{equation}
Of course we can estimate the entire probability distribution on $X$ by computing $\tilde{p}_j$ for  $j= 1, \dots , n$.

\medskip

Here is an example. 

\begin{example}
\label{n=2.exam}
We use the designs from Example \ref{E1.exam}.

Suppose $\Theta = p^* = 3/4$.
Then we have 
$q^* =  {5}/{12}$ from (\ref{q*2.eq}).
Suppose we have $t = 18$ samples (consisting of values $a,b,c,d,e,f$), where 
$f_\MS{a} = f_\MS{b} =  4$, $f_\MS{c} = f_\MS{d} = 2$ and $f_\MS{e} = f_\MS{f} = 3$.
Then $T_1 = 4 + 4 + 2 = 10$. 

Our estimate of $\tilde{p}_1$ is computed using (\ref{estimator2}) as follows:
\begin{align*} \tilde{p}_1 &= \frac{10 - 18(\frac{5}{12})}{18(\frac{3}{4} - \frac{5}{12})}\\
&= \frac{10 - \frac{15}{2}}{18(\frac{1}{3})}\\
&= \frac{5}{12}.
\end{align*}
In a similar fashion, we can compute
$\tilde{p}_2 = \tilde{p}_3 = 1/4$ and $\tilde{p}_4 = 1/12$.
\end{example}



Estimators can be investigated in a more general setting, e.g., as is done in \cite{CN}.
Suppose we start with $Q$, which is an \emph{arbitrary} $m$ by $n$ TPM with $m \geq n$. We assume that $\mathsf{rank}(Q) = n$ (i.e., the columns of $Q$ are linearly independent). It is fairly easy to show that an unbiased linear estimator\footnote{We defined unbiased estimators in Section \ref{intro.sec}. An estimator is \emph{linear} if it is a linear function of its inputs.} for $Q$ is any $n$ by $m$ matrix $L$ such that $LQ = I_{n \times n}$ (i.e., $L$ is a \emph{left inverse} of $Q$). (Note that the desired matrix $L$ cannot exist if $\mathsf{rank}(Q) < n$.) If $n = m$, then $L$ is unique and $L = Q^{-1}$. However, if $n < m$, then there are many possible left inverses of $Q$ and hence there are many possible linear estimators. 

The estimator defined in (\ref{estimator2}) can  easily be written in matrix form. 
For $1 \leq i \leq m$, 
denote $\mathbf{f} = (f_1, \dots , f_m)$ and let  
$\hat{\rho} = (\frac{1}{t} \mathbf{f})^T$. Note that $\hat{\rho}$ is a column vector and $\hat{\rho}_i = f_i /t$ for all $i$.
Now we use  the fact that
$\sum f_i = t$ to rewrite (\ref{estimator2}):
\begin{align}
\nonumber \tilde{p}_j  &= \frac{\sum_{y_i \in Y_j} f_i -  q^* \sum_{i=1}^m f_i}{t(p^* - q^*)}\\
\nonumber &= \frac{\sum_{y_i \in Y_j} (1 - q^*)f_i -   \sum_{y_i \not\in Y_j} q^*f_i}{t(p^* - q^*)}\\
\label{Lformula.eq} &= \frac{\sum_{y_i \in Y_j} (1 - q^*)\hat{\rho}_i -   \sum_{y_i \not\in Y_j} q^* \hat{\rho}_i}{p^* - q^*}.
\end{align}
Define
\begin{align}
\label{gamma1.eq} \gamma_1 &= \frac{1 - q^*}{p^* - q^*}\\
\label{gamma2.eq} \gamma_2 &= \frac{-q^*}{p^* - q^*}.
\end{align}
Then it is clear from (\ref{Lformula.eq}) that the associated estimator is
\begin{equation}
\label{L.eq}
L = \gamma_1 A^T + \gamma_2 (J_{v \times b} - A^T) = A^T (\gamma_1 - \gamma_2) + \gamma_2 J_{v \times b}.
\end{equation}

\begin{example}
Suppose $A$ is the following $12$ by $9$ incidence matrix of a $(9,12,4,3,1)$-BIBD:
\[A = 
\left( 
\begin{array}{ccccccccc}
1& 1& 1& 0& 0& 0& 0& 0& 0\\
0& 0& 0& 1& 1& 1& 0& 0& 0\\
0& 0& 0& 0& 0& 0& 1& 1& 1\\
1& 0& 0& 1& 0& 0& 1& 0& 0\\
0& 1& 0& 0& 1& 0& 0& 1& 0\\
0& 0& 1& 0& 0& 1& 0& 0& 1\\
1& 0& 0& 0& 1& 0& 0& 0& 1\\
0& 1& 0& 0& 0& 1& 1& 0& 0\\
0& 0& 1& 1& 0& 0& 0& 1& 0\\
1& 0& 0& 0& 0& 1& 0& 1& 0\\
0& 1& 0& 1& 0& 0& 0& 0& 1\\
0& 0& 1& 0& 1& 0& 1& 0& 0
\end{array}
\right).
\] 
If we set $\theta = 3/4$, then $\alpha_1 = 3/16$ and $\alpha_2 = 1/32$.  Note that the differential privacy ratio $\alpha_1/ \alpha_2 = 6$, so $\epsilon = \ln 6$.
The resulting TPM is 
\[Q = 
\left( 
\begin{array}{ccccccccc}
\frac{3}{16} &  \frac{3}{16} &  \frac{3}{16} &  \frac{1}{32} &  \frac{1}{32} &  \frac{1}{32} &  \frac{1}{32} &  \frac{1}{32} &  \frac{1}{32} \vspace{.05in}\\ 
\frac{1}{32} &  \frac{1}{32} &  \frac{1}{32} &  \frac{3}{16} &  \frac{3}{16} &  \frac{3}{16} &  \frac{1}{32} &  \frac{1}{32} &  \frac{1}{32}\vspace{.05in}\\ 
\frac{1}{32} &  \frac{1}{32} &  \frac{1}{32} &  \frac{1}{32} &  \frac{1}{32} &  \frac{1}{32} &  \frac{3}{16} &  \frac{3}{16} &  \frac{3}{16}\vspace{.05in}\\ 
\frac{3}{16} &  \frac{1}{32} &  \frac{1}{32} &  \frac{3}{16} &  \frac{1}{32} &  \frac{1}{32} &  \frac{3}{16} &  \frac{1}{32} &  \frac{1}{32}\vspace{.05in}\\ 
\frac{1}{32} &  \frac{3}{16} &  \frac{1}{32} &  \frac{1}{32} &  \frac{3}{16} &  \frac{1}{32} &  \frac{1}{32} &  \frac{3}{16} &  \frac{1}{32}\vspace{.05in}\\ 
\frac{1}{32} &  \frac{1}{32} &  \frac{3}{16} &  \frac{1}{32} &  \frac{1}{32} &  \frac{3}{16} &  \frac{1}{32} &  \frac{1}{32} &  \frac{3}{16}\vspace{.05in}\\ 
\frac{3}{16} &  \frac{1}{32} &  \frac{1}{32} &  \frac{1}{32} &  \frac{3}{16} &  \frac{1}{32} &  \frac{1}{32} &  \frac{1}{32} &  \frac{3}{16}\vspace{.05in}\\ 
\frac{1}{32} &  \frac{3}{16} &  \frac{1}{32} &  \frac{1}{32} &  \frac{1}{32} &  \frac{3}{16} &  \frac{3}{16} &  \frac{1}{32} &  \frac{1}{32}\vspace{.05in}\\ 
\frac{1}{32} &  \frac{1}{32} &  \frac{3}{16} &  \frac{3}{16} &  \frac{1}{32} &  \frac{1}{32} &  \frac{1}{32} &  \frac{3}{16} &  \frac{1}{32}\vspace{.05in}\\ 
\frac{3}{16} &  \frac{1}{32} &  \frac{1}{32} &  \frac{1}{32} &  \frac{1}{32} &  \frac{3}{16} &  \frac{1}{32} &  \frac{3}{16} &  \frac{1}{32}\vspace{.05in}\\ 
\frac{1}{32} &  \frac{3}{16} &  \frac{1}{32} &  \frac{3}{16} &  \frac{1}{32} &  \frac{1}{32} &  \frac{1}{32} &  \frac{1}{32} &  \frac{3}{16}\vspace{.05in}\\ 
\frac{1}{32} &  \frac{1}{32} &  \frac{3}{16} &  \frac{1}{32} &  \frac{3}{16} &  \frac{1}{32} &  \frac{3}{16} &  \frac{1}{32} &  \frac{1}{32}
\end{array}
\right).
\]
We have $p^* = 3/4$ and $q^* = 9/32$ from Theorem \ref{BIBD.thm}. Then $\gamma_1 = 23/15$ and $\gamma_2 = -3/5$ from (\ref{gamma1.eq}) and (\ref{gamma2.eq}). The estimator given by (\ref{L.eq}) is 
\[
L = 
\left( 
\begin{array}{rrrrrrrrrrrr}
\frac{23}{15} &  -\frac{3}{5} &  \frac{3}{5} &  \frac{23}{15} &  -\frac{3}{5} &  -\frac{3}{5} &  \frac{23}{15} &  -\frac{3}{5} &  -\frac{3}{5} &  \frac{23}{15} &  -\frac{3}{5} &  -\frac{3}{5}\vspace{.05in}\\ 
\frac{23}{15} &  -\frac{3}{5} &  -\frac{3}{5} &  -\frac{3}{5} &  \frac{23}{15} &  -\frac{3}{5} &  -\frac{3}{5} &  \frac{23}{15} &  -\frac{3}{5} &  -\frac{3}{5} &  \frac{23}{15} &  -\frac{3}{5}\vspace{.05in}\\ 
\frac{23}{15} &  -\frac{3}{5} &  -\frac{3}{5} &  -\frac{3}{5} &  -\frac{3}{5} &  \frac{23}{15} &  -\frac{3}{5} &  -\frac{3}{5} &  \frac{23}{15} &  -\frac{3}{5} &  -\frac{3}{5} &  \frac{23}{15}\vspace{.05in}\\ 
-\frac{3}{5} &  \frac{23}{15} &  -\frac{3}{5} &  \frac{23}{15} &  -\frac{3}{5} &  -\frac{3}{5} &  -\frac{3}{5} &  -\frac{3}{5} &  \frac{23}{15} &  -\frac{3}{5} &  \frac{23}{15} &  -\frac{3}{5}\vspace{.05in}\\ 
-\frac{3}{5} &  \frac{23}{15} &  -\frac{3}{5} &  -\frac{3}{5} &  \frac{23}{15} &  -\frac{3}{5} &  \frac{23}{15} &  -\frac{3}{5} &  -\frac{3}{5} &  -\frac{3}{5} &  -\frac{3}{5} &  \frac{23}{15}\vspace{.05in}\\ 
-\frac{3}{5} &  \frac{23}{15} &  -\frac{3}{5} &  -\frac{3}{5} &  -\frac{3}{5} &  \frac{23}{15} &  -\frac{3}{5} &  \frac{23}{15} &  -\frac{3}{5} &  \frac{23}{15} &  -\frac{3}{5} &  -\frac{3}{5}\vspace{.05in}\\ 
-\frac{3}{5} &  -\frac{3}{5} &  \frac{23}{15} &  \frac{23}{15} &  -\frac{3}{5} &  -\frac{3}{5} &  -\frac{3}{5} &  \frac{23}{15} &  -\frac{3}{5} &  -\frac{3}{5} &  -\frac{3}{5} &  \frac{23}{15}\vspace{.05in}\\ 
-\frac{3}{5} &  -\frac{3}{5} &  \frac{23}{15} &  -\frac{3}{5} &  \frac{23}{15} &  -\frac{3}{5} &  -\frac{3}{5} &  -\frac{3}{5} &  \frac{23}{15} &  \frac{23}{15} &  -\frac{3}{5} &  -\frac{3}{5}\vspace{.05in}\\ 
-\frac{3}{5} &  -\frac{3}{5} &  \frac{23}{15} &  -\frac{3}{5} &  -\frac{3}{5} &  \frac{23}{15} &  \frac{23}{15} &  -\frac{3}{5} &  -\frac{3}{5} &  -\frac{3}{5} &  \frac{23}{15} &  -\frac{3}{5}\vspace{.05in}
\end{array}
\right).
\]
\end{example}

\subsection{Estimators and Moore-Penrose Left Inverses}

Suppose $Q$ is an $m$ by $n$ matrix (where $m > n$) with linearly independent columns. Among the left inverses of $Q$  is a special one known as the \emph{Moore-Penrose inverse}. The Moore-Penrose inverse of $Q$ is denoted by $Q^+$ and it is given by the following formula:
\begin{equation}
\label{MPI.eq}
Q^+ = (Q^T Q)^{-1} Q^T
.\end{equation} 
It is not hard to see that $Q^+ Q = I_{n \times n}$.

We will prove that the estimator given by equation (\ref{L.eq}) is in fact just the Moore-Penrose inverse $Q^+$ when $Q$ is the TPM arising from a BIBD. 

We state three useful lemmas. 
\begin{lemma}
\label{properties.lem}
Suppose $A$ is the $b$ by $v$ incidence matrix of a 
$(v,b,r,k,\lambda)$-BIBD. Then 
\begin{align*}
A^T A &= (r - \lambda)I_{v \times v} + \lambda J_{v \times v} \\
A^T J_{b \times v} &= J_{v \times b}A =  r J_{v \times v}\\
J_{v \times v} A^T  &= k J_{v \times b}\\
J_{v \times b}J_{b \times v} &= b J_{v \times v}.
\end{align*}
\end{lemma}
\begin{proof}
The formula for $A^T A$ is just equation (\ref{BIBD.eq}). The remaining formulas are straightforward to prove.
\end{proof}


\begin{lemma}
\label{QQT.eq}
Suppose $A$ is the $b$ by $v$ incidence matrix of an $(r, \lambda)$-design (which includes the case of a 
$(v,b,r,k,\lambda)$-BIBD) and $Q = (\alpha_1 - \alpha_2) A_{b \times v} + \alpha_2 J_{b \times v}$.
Then
\begin{equation}
\label{Q.eq} Q^T Q = c I_{v \times v} + d J_{v \times v},
\end{equation}
where
\begin{align*}
c &= (r - \lambda)(\alpha_1 - \alpha_2)^2 \\
d &= \lambda (\alpha_1 - \alpha_2)^2 + 2r \alpha_2(\alpha_1 - \alpha_2) + {\alpha_2}^2 b.
\end{align*}
\end{lemma}

\begin{proof}
The equation (\ref{Q.eq}) is easily proven by expanding the expression $Q^T Q$ using Lemma \ref{properties.lem}.
\end{proof}

\begin{lemma}
\label{inverse2.lem}
If $c \neq 0$  and $vd + c\neq 0$, then
\begin{equation}
\label{inverse2.eq}
(c I_{v \times v} + d J_{v \times v})^{-1} = c' I_{v \times v} + d' J_{v \times v},
\end{equation}
where
\[
c' = \frac{1}{c} 
\quad \text{and} \quad d' = -\frac{d}{c(vd + c)}.
\]
\end{lemma}
\smallskip

\begin{theorem}
\label{MP.thm}
Suppose the TPM $Q$ is based on the $b$ by $v$ incidence matrix $A$ of a 
$(v,b,r,k,\lambda)$-BIBD. Then the  matrix $L$ defined in (\ref{L.eq}) is the Moore-Penrose inverse of $Q$.
\end{theorem}

\begin{proof}
The proof has three steps:
\begin{enumerate}
\item The matrix $L$ is a linear combination of $A^T$ and $J_{v \times b}$.
\item The Moore-Penrose left inverse of $Q$ is a linear combination of $A^T$ and $J_{v \times b}$.
\item The matrix $Q$ has a unique left inverse that is a linear combination of $A^T$ and $J_{v \times b}$.
\end{enumerate}
We note that step 1 follows immediately from (\ref{L.eq}).

To prove step 2, it follows from (\ref{MPI.eq}) that we need to consider the structure of 
$(Q^T Q)^{-1} Q^T$.
First, from equation (\ref{Q.eq}) in Lemma \ref{QQT.eq}, we have that $Q^T Q$ is a linear combination of 
$I_{v \times v}$ and $J_{v \times v}$.
Then, from Lemma \ref{inverse2.lem}, $(Q^T Q)^{-1}$ is also a linear combination of $I_{v \times v}$ and 
$J_{v \times v}$. It is clear from (\ref{TPM.eq}) that $Q^T$ is a linear combination of $A^T$ and 
$J_{v \times b}$.
Therefore, from Lemma \ref{properties.lem}, the matrix $Q^+ = (Q^T Q)^{-1} Q^T$
is a linear combination of $A^T$ and $J_{v \times b}$.

To prove step 3, denote $Q = \alpha A + \beta J_{b \times v}$ consider the equation 
\[ (\gamma A^T + \delta J_{v \times b})Q = I_{v\times v}.\]
Using Lemma \ref{properties.lem}, we have the following:
\begin{align*}
(\gamma A^T + \delta J_{v \times b})Q &= (\gamma A^T + \delta J_{v \times b})(\alpha A + \beta J_{b \times v})\\
&= \alpha \gamma A^TA + \beta \gamma A^T J_{b \times v} + \alpha \delta J_{v \times b} A +
\beta \delta J_{v \times b}J_{b \times v}\\
&= \alpha \gamma ((r - \lambda)I_{v \times v} + \lambda J_{v \times v}) +
r (\beta \gamma + \alpha \delta) J_{v \times v} + \beta \delta b J_{v \times v}\\
&= \alpha \gamma (r - \lambda)I_{v \times v} + (\alpha \gamma \lambda + r (\beta \gamma + \alpha \delta) + \beta \delta b) J_{v \times v}.
\end{align*}
It follows that $\gamma A^T + \delta J_{v \times b}$ is a left inverse of $Q$ if and only if
\begin{align*}
\alpha \gamma (r - \lambda) &= 1 \quad \text{and}\\
\alpha \gamma \lambda + r  (\beta \gamma + \alpha \delta) + \beta \delta b &= 0.
\end{align*}
Given $\alpha$ and $\beta$, we have two linear equations in the unknowns $\gamma$ and $\delta$.
We can solve for $\gamma$ in the first equation:
\[ \gamma = \frac{1}{\alpha(r-\lambda)}.\]
 Then substitute $\gamma$ into the second equation to solve for $\delta$:
 \[\delta = \frac{- \gamma ( \alpha \lambda + r\beta)}{\alpha r + \beta b}.\] 
 This proves step 3. Steps 1--3 immediately yield the desired conclusion.
 \end{proof}

 \begin{remark}
Theorem \ref{MP.thm} does not necessarily hold if $A$ is the incidence matrix of an $(r,\lambda)$-design that is not a BIBD. This because the equation $J_{v \times v} A^T  = k J_{v \times b}$ stated in Lemma \ref{properties.lem} only makes sense if $k$ is constant. So step 2 of the proof of Theorem \ref{MP.thm} is invalid if we do not start with a BIBD.
\end{remark}

 

\section{Computing the Variance}
\label{variance.sec}

Now we discuss the computation of the variance associated with an estimator. 
We wish to determine the distribution
of the random variable $T_j$ in order to help us compute the variance of
the estimator $\tilde{p_j}$ given by the formula (\ref{estimator}).
We observe that there are $t$ output samples, each of which contributes either $+1$ or $+0$ to the value of $T_j$.  These samples are generated independently, so if we can determine the probability $p$ that a given sample contributes $+1$ then we know that  $T_j$ has a binomial distribution $B(t,p)$ with the specified probability $p$.  We can determine $p$ for a scheme based on an $(r,\lambda)$-design as follows:

\begin{itemize}
\item With probability $p_j$, the input sample used to generate the output sample is $x_j$.  The probability that this occurs and the output contributes $+1$ to $T_j$ is $p_jp^*$.
\item With probability $1-p_j$, the input sample is not $x_j$.  the probability that this occurs and the output contributes $+1$ to $T_j$ is $(1-p_j)q^*$.
\end{itemize}
Thus the probability that a given input sample results in an output sample that contributes $+1$ to $T_j$ is \begin{align*}
p &= p_jp^*+(1-p_j)q^*\\
&=p_j(p^*-q^*)+q^*.
\end{align*}

So we have $T_j\sim B(t,p_j(p^*-q^*)+q^*)$, and hence the variance of $T_j$ is
\begin{align*}
  \operatorname{Var}(T_j)&= tp(1-p)\\
  &= t(p_j(p^*-q^*)+q^*)(1-(p_j(p^*-q^*)+q^*)).
\end{align*}

Now we compute
\begin{align}
\operatorname{Var}(\tilde{p}_j)&=\operatorname{Var}\left(\frac{T_j-tq^*}{t(p^*-q^*)}\right) \nonumber\\
&=\frac{\operatorname{Var}(T_j)}{t^2(p^*-q^*)^2}\nonumber\\
&=\frac{t(p_j(p^*-q^*)+q^*)(1-(p_j(p^*-q^*)+q^*))}{t^2(p^*-q^*)^2}\nonumber\\
&=\frac{(p_j(p^*-q^*)+q^*)(1-(p_j(p^*-q^*)+q^*))}{t(p^*-q^*)^2}\nonumber\\
&=\frac{(1-2q^*)p_j(p^*-q^*)+q^*-p_j^2(p^*-q^*)^2-(q^*)^2}{t(p^*-q^*)^2}.\label{varM.eq}
\end{align}

To get 
variance of the estimator, we sum this expression (\ref{varM.eq}) for $j=1,2,\dotsc ,n$.
Denote $\tilde{\pi} = (\tilde{p}_1, \dots , \tilde{p}_n)$, so 
\[\operatorname{Var}(\tilde{\pi}) = \sum_{j=1}^{n} \operatorname{Var}(\tilde{p}_j).\]
Then we have 
\begin{align}
\nonumber 
\operatorname{Var}(\tilde{\pi})
&=\sum_{j=1}^n\frac{(1-2q^*)p_j(p^*-q^*)+q^*-p_j^2(p^*-q^*)^2-(q^*)^2}{t(p^*-q^*)^2},\\
\nonumber &=\frac{(1-2q^*)}{t(p^*-q^*)}+\sum_{j=1}^n\frac{q^*-p_j^2(p^*-q^*)^2-(q^*)^2}{t(p^*-q^*)^2},\\
\label{El2.eq}&=\frac{1-2q^*}{t(p^*-q^*)}+\frac{nq^*(1 - q^*)}{t(p^*-q^*)^2}-\sum_{j=1}^n\frac{p_j^2}{t}.
\end{align}
Note that the above equation holds for any $(r,\lambda)$-design (including BIBDs).

\medskip

We mentioned in the Introduction that Warner \cite{War} introduced the idea of ``randomized response'' in 1965. This corresponds to the case of 
$n=2$ possible outcomes. 
We explore estimators in this setting now.

\begin{example}
We consider the  set systems $(X,\mathcal{B})$ and $(Y,\mathcal{Y})$,
where $X = \{1,2\}$, $B_{\MS{a}} = \{1\}$ and $B_{\MS{b}} = \{2\}$; and $Y = \{\MS{a},\MS{b}\}$, $Y_1 = \{\MS{a}\}$ and $Y_2 = \{\MS{b}\}$.
So we have $m=n=2$. These designs are both (trivial) $(2,2,1,1,0)$-BIBDs.

For a given $\theta$, we have $p^* = \theta$ and $q^* = 1 - \theta$; hence, $p^* - q^* = 2 \theta - 1$.
We compute the estimator $\tilde{p}_1$.
From (\ref{estimator}), we have 
\[\tilde{p}_1 = \frac{f_1 -t (1- \theta)}{t(2 \theta - 1)}.\]
Of course, we have $E[\tilde{p}_1] = p_1$. First, we compute $\mathrm{Var}(\tilde{p}_1)$ using
the formula (6) from \cite{War}:
\begin{align*}
\mathrm{Var}(\tilde{p}_1) &= \frac{1}{t} \left( \frac{1}{4(2\theta-1)^2} 
- \left( p_1 - \frac{1}{2} \right)^2 \right)\\
&= \frac{1}{t} \left( \frac{1}{4(2\theta-1)^2} 
- \left( {p_1}^2 - p_1 + \frac{1}{4} \right)
\right)\\
&= \frac{1}{t} \left( p_1 - {p_1}^2  
- \frac{1}{4} \left( 1 -   \frac{1}{(2\theta-1)^2} \right)
\right)\\
&= \frac{1}{t} \left( p_1 - {p_1}^2  
+ \frac{\theta - \theta^2}{(2\theta-1)^2} 
\right)
.\end{align*}

To compare, we repeat the calculation using (\ref{varM.eq}):
\begin{align*}
\mathrm{Var}(\tilde{p}_1) &= 
\frac{q^*(1 - q^*) + (1-2q^*)(p^*-q^*)p_1 -(p^*-q^*)^2 {p_1}^2}{t(p^*-q^*)^2}\\
&= \frac{\theta - \theta^2 + (2\theta-1)^2 p_1 - (2\theta-1)^2 {p_1}^2 }{t(2\theta-1)^2}\\
&= \frac{1}{t} \left( \frac{\theta - \theta^2}{(2\theta-1)^2} + p_1 - {p_1}^2 \right) 
.\end{align*}
As expected, we get the same value for the variance in both cases.
\end{example} 

\bigskip

When we start with a BIBD, our formula (\ref{El2.eq}) is identical to equation (33) from \cite{PNL}, 
 which is stated as follows:
\begin{equation}
\label{PNL.eq} 
\operatorname{Var}(\tilde{\pi}) 
= \frac{1}{t} \left( \frac{(n-1)^2(k e^{\epsilon} + n - k)^2}{k(n-k)(e^{\epsilon}-1)^2 n}  + \frac{1}{n} - \sum_{j=1}^n {p_j^2}\right).
\end{equation}

There is a more complicated expression, presented as equation (35) in \cite{PNL}, which gives a formula
$\operatorname{Var}(\tilde{\pi})$ 
in the case where the LDP protocol is based on an $(r, \lambda)$-design:
\begin{equation}
\label{PNL2.eq} 
\operatorname{Var}(\tilde{\pi}) 
= \frac{1}{t} \left( \frac{(r e^{\epsilon} + (n-1)(\lambda e^{\epsilon} + r - \lambda))%
(n(m-r)+ (n-1)(r - \lambda)(e^{\epsilon}-1))}%
{(r -\lambda)^2 (e^{\epsilon}-1)^2 n}  
+ \frac{1}{n} - \sum_{j=1}^n {p_j^2}\right).
\end{equation}
 
\medskip
 
 Equation (\ref{PNL2.eq})
of course contains the BIBD formula  (\ref{PNL.eq}) as a special case, though this might might not be immediately obvious by inspecting the two formulas.
Further, we note that equation (\ref{PNL2.eq})  is equivalent to our formula (\ref{El2.eq}); this is also not obvious.

\subsection{Optimality of Estimators}
\label{optimal.sec}

There has been considerable work concerning optimality of estimators. 
See, for example, \cite{CN,Gen,PNL,Wang2019,YB,YB19}.
It has been shown by various authors that estimators derived from BIBD-based LDP protocols are optimal, under certain assumptions. A recent work by Gentle 
\cite{Gen} establishes that  all optimal estimators are derived from BIBDs. In this section, we mainly follow the approach of Chai and Nayak \cite{CN}.

We use the following notation  that we have already introduced. $Q$ is the TPM, $\pi = (p_1, \dots , p_n)$ is the probability distribution on $X$,
and $\rho = (\rho_1, \dots , \rho_m)$ is the induced distribution on $Y$.
Also, $\tilde{\pi} = (\tilde{p_1}, \dots , \tilde{p_n})$ is the probability distribution on $X$ given by the estimator. 
$D_{\rho}$ is the matrix with values $\rho_1, \dots , \rho_m$ on the diagonal and $0$'s elsewhere, and $t$ denotes the number of samples of elements from $X$. 

Later, we will make use of the following simple lemma, which is very similar to \cite[Lemma 4]{Gen}.

\begin{lemma}
\label{uniform.lem}
Suppose $Q$ is a TPM based on a $(v,b,r,k,\lambda)$-BIBD and $\pi$ is the uniform distribution on $X$, where $n = v$.
Then $\rho = (\rho_1, \dots , \rho_m)$ is the uniform distribution on $Y$, where $m = b$.
\end{lemma}

\begin{proof}
We  compute 
\[ \rho = Q \left(\frac{1}{n},  \dots , \frac{1}{n} \right)^T,\]
where $n = v$.
Note that we write $\rho$ as a column vector, for convenience.
Since $(X,\mathcal{B})$ is a BIBD, it follows that $A$ has constant row sum, namely $k$. Then, from (\ref{TPM.eq}), $Q$ also has constant row sum. Therefore $\rho_1 = \dots = \rho_m$. Since $\rho$ is a probability distribution, we must have
$\rho_i = 1/m$ for $i = 1, \dots , m$.
\end{proof}

\begin{remark}
Lemma \ref{uniform.lem} does not hold if $Q$ is a TPM based on an $(r,\lambda)$-design that does not have constant block size.
\end{remark}

One general bound, given by Chai and Nayak in \cite[Proposition 2.1]{CN}, is stated as follows:
\begin{equation}
\label{CN.eq}
\operatorname{Var}(\tilde{\pi}) \geq \frac{1}{t}\left( \mathsf{trace}((Q^T D_{\rho}^{-1} Q)^{-1}) - \sum_{j=1}^n {p_j^2}\right) .
\end{equation}

Prior to proving (\ref{CN.eq}), it is first stated without proof in \cite[equation (2.1)]{CN} that
\begin{equation}
\label{CN1.eq}
\operatorname{Var}(\tilde{\pi}) = \frac{1}{t}\left( \mathsf{trace}(L D_{\rho} L^T) - \sum_{j=1}^n {p_j^2}\right) .
\end{equation}
In (\ref{CN1.eq}), the estimator $L$ is any left inverse of the TPM $Q$.

\medskip

We  discuss some details of the derivation of (\ref{CN1.eq}).
It is clear that $\rho = Q \pi^T$ is the probability distribution on  $Y$ that is induced by the distribution $\pi = (p_1, \dots , p_n)$ on  $X$ and the TPM $Q$.  As usual, we are assuming that $t$ elements of $X$ are drawn independently according to the distribution $\pi$ then encoded as elements of $Y$ using $Q$.

For $1 \leq j \leq m$, let $f_j$ be the observed number of occurrences of $y_j$ among the $t$ samples. 
Denote $\mathbf{f} = (f_1, \dots , f_m)$ and let  
\begin{equation}
\label{rhohat.eq}
\hat{\rho} = \left( \frac{1}{t} \mathbf{f} \right)^T.
\end{equation} The column vector $\hat{\rho}$ is an unbiased estimator for $\rho$.


An unbiased linear estimator for $\pi$  is a linear function of $\hat{\rho}$ that can be described as multiplication of $\hat{\rho}$ by an $n\times m$ matrix $L$.  Thus we estimate the distribution $\pi$ to be
$\tilde{\pi}$, where 
\[\tilde{\pi}^T = L \hat{\rho}.\] This estimator $L$ is unbiased for any distribution $\pi$ on $X$ precisely when $LQ=I_{n\times n}$.
 
The {\em risk function} of an estimator $L$ is the quantity
\[ \mathbf{R}(Q,L;\pi) = tE\left[\Vert \tilde{\pi}-\pi \Vert ^2\right].\] Thus
$\mathbf{R}(Q,L;\pi)$ is the expected value of the square of the euclidean distance between $\tilde{\pi} = (L\hat{\rho})^T$ and $\pi$, multiplied by $t$ (in order to normalize the risk for sample size). 
By definition of the euclidean distance and the linearity of expectation, we have
\begin{align}
\nonumber \mathbf{R}(Q,L;\pi) &=
t\, E\left[\Vert \tilde{\pi}-\pi\Vert ^2\right]\\
\nonumber &=t\sum_{i=1}^n E\left[(\tilde{p_i}-p_i)^2\right],\\
\nonumber 
&=t\sum_{i=1}^n\operatorname{Var}(\tilde{p_i}) ,
\end{align}
since $p_i = E[\tilde{p_i}]$.

Let $V = (v_{ij})$ denote the {covariance matrix} of $L\hat{\rho}$. 
Thus $V$ is the matrix having entries 
\[v_{ij} = E\left[(\tilde{p_i}-p_i)(\tilde{p_j}-p_j)\right].\]  
In particular, 
$v_{ii} = \operatorname{Var}(\tilde{p_i})$. 
We therefore have
\begin{align*}
t\sum_{i=1}^n\operatorname{Var}(\tilde{p_i})&=t\sum_{i=1}^n v_{ii}\\
&=t\, \mathsf{trace}(V).
\end{align*}

From (\ref{rhohat.eq}), we see that $\hat{\rho}$ is the mean of $t$ independent and identically distributed (i.i.d.) samples, and 
each $\tilde{p_i}$ is a linear combination of entries of $\hat{\rho}$.  It follows that 
\[\operatorname{Var}(\tilde{p_i})= \frac{1}{t}\, \operatorname{Var}((L\hat{\rho}^\prime)_i),\] where $\hat{\rho}^\prime$ is the outcome of a single sample, so we have 
\begin{align}
t\, \mathsf{trace}(V)&=\sum_{i=1}^n\operatorname{Var}((L\hat{\rho}^\prime)_i).
\end{align}

Observe that, by definition, $\mathsf{Prob}[Y = y_j] = \rho_j$. 
We therefore have 
\begin{align}
\nonumber \operatorname{Var}((L\hat{\rho^\prime})_i )&=E\left[(L\hat{\rho}^\prime)_i^2\right]-p_i^2,\\
\label{varLrhoprime.eq} &=\left(\sum_{j=1}^m\rho_jL_{ij}^2\right)-p_i^2.
\end{align}

Recall that $D_\rho$ is the $m\times m$ diagonal matrix with the entries of $\rho_1, \dots , \rho_m$ on the main diagonal.  We aim to show that 
\begin{equation}
\label{risk.eq} \mathbf{R}(Q,L;\pi) = \mathsf{trace}(LD_\rho L^T)-\sum_{i=1}^np_i^2.
\end{equation} 
We note that
\begin{align*}
(LD_\rho)_{ij}&=\sum_{k=1}^m L_{ik}(D_\rho)_{kj},\\
&=\rho_jL_{ij}.
\end{align*}
Thus
\begin{align*}
(LD_\rho L^T)_{ij}&=\sum_{k=1}^m(LD_\rho)_{ik}L^T_{kj}\\
&=\sum_{k=1}^m\rho_kL_{ik}L_{jk}.
\end{align*}
In particular, \[(LD_\rho L^T)_{ii}=\sum_{k=1}^m\rho_kL_{ik}^2.\]
From this we see that
\begin{align*}
\mathsf{trace}(LD_\rho L^T)-\sum_{i=1}^np_i^2&=\sum_{i=1}^n\left(\sum_{k=1}^m \rho_kL_{ik}^2-p_i^2\right),\\
&=\sum_{i=1}^n\operatorname{Var}((L\hat{\rho}^\prime)_i) \quad \text{from (\ref{varLrhoprime.eq})},
\end{align*}
and so we conclude that 
(\ref{risk.eq}) holds. 
 
 To recapitulate, the formula (\ref{CN1.eq}) expresses the variance (or risk) as the trace of a certain matrix that depends on the estimator and the induced probability distribution on $Y$. However, we have noted that there are many possible unbiased linear estimators for a given TPM when $m > n$. It is natural to ask which estimator yields the lowest risk for a given TPM. Chai and Nayak \cite{CN} show that a lower bound for this risk is given by the inequality (\ref{CN.eq}). They show further in \cite[Proposition 2]{CN} that this bound (\ref{CN.eq}) is met with equality when the estimator $L$ is given by the following formula:
\begin{equation}
\label{CN-L.eq} L = (Q^T (D_\rho)^{-1} Q)^{-1} Q^T (D_\rho)^{-1}.
\end{equation}
 
 \bigskip
 
 For the rest of this section, suppose $Q$ is a BIBD-based TPM and $\pi$ is a uniform probability distribution on $X$. Then, from Lemma \ref{uniform.lem}, the induced probability distribution on $Y$ is also uniform: $\rho_i = 1/m$ for $1 \leq i \leq m$. The following theorem is an immediate consequence.
 
 \begin{theorem}
 Suppose $Q$ is a BIBD-based TPM and $\pi$ is a uniform probability distribution on $X$.
Then the (optimal) estimator $L$ given by (\ref{CN-L.eq}) is the Moore-Penrose left inverse of $Q$.
\end{theorem}
 
 \begin{proof}
 We have $D_\rho = \frac{1}{m} I_{m \times m}$ and hence 
 \[ (Q^T (D_\rho)^{-1} Q)^{-1} Q^T (D_\rho)^{-1} = (Q^T  Q)^{-1} Q^T .\] 
 \end{proof}
  
 If we assume a uniform probability distribution on $X$, then the bound (\ref{CN.eq}) simplifies to give the following:
\begin{equation}
\label{CN2.eq}
\operatorname{Var}(\tilde{\pi}) \geq \frac{1}{t}\left( \frac{\mathsf{trace}((Q^T Q)^{-1})}{m} - \frac{1}{n} \right) .
\end{equation}
Further, equality is attained for a BIBD-based TPM and a uniform probability distibution on $X$.
 
\medskip

We present a numerical example.

\begin{example}
Suppose $(X, \mathcal{B})$ is the $(9,12,4,3,1)$-BIBD (an affine plane of order three). The blocks are
\[
\begin{array}{ccc}
\{1,2,3\} & \{4,5,6\} & \{7,8,9\}\\
\{1,4,7\} & \{2,5,8\} & \{3,6,9\}\\
\{1,5,9\} & \{2,6,7\} & \{3,4,8\}\\
\{1,6,8\} & \{2,4,9\} & \{3,5,7\}
\end{array}
\]
Suppose we have an equiprobable source distribution, i.e., $\mathsf{Prob}[ X = x ] = 1/9$ for $1 \leq x \leq 9$. Suppose also that $\theta = 1/2$ and $t = 10$. The value
of $\operatorname{Var}(\tilde{\pi})$ is computed to be $256/45$ using both (\ref{El2.eq}) and (\ref{PNL.eq}).
It is also the case that the right side of both (\ref{CN.eq}) and (\ref{CN2.eq}) has the value $256/45$.
%
\end{example}

We can simplify the lower bound in (\ref{CN2.eq}) by computing the desired trace,
assuming that we have an LDP protocol based on a BIBD and the probability distribution on $X$ is uniform. 
From (\ref{Q.eq}) and (\ref{inverse2.eq}), we have
\begin{equation}
\label{QTQ.eq}
 (Q^T Q)^{-1} = c' I_{v \times v} + d' J_{v \times v},
 \end{equation}
where
\begin{equation}\label{4eq}
\begin{split}
c &= (r - \lambda)(\alpha_1 - \alpha_2)^2 \\
d &= \lambda (\alpha_1 - \alpha_2)^2 + 2r \alpha_2(\alpha_1 - \alpha_2) + {\alpha_2}^2 b\\
c' &= \frac{1}{c}\\
d' &= -\frac{d}{c(vd + c)}.
\end{split}
\end{equation}

Therefore  we have  
\begin{equation}
\label{trace.eq}
\mathsf{trace}((Q^T  Q)^{-1}) = v(c' + d'),
\end{equation}
where $c'$ and $d'$ are given  in (\ref{4eq}).

Therefore, the bound (\ref{CN2.eq}) can be rewritten as 
\[
 \operatorname{Var}(\tilde{\pi}) \geq \frac{1}{t}\left( \frac{v(c' + d')}{b} - \frac{1}{v}\right).
\]
Also, as before, equality occurs in this bound when we have a TPM based on a BIBD and the distribution on $X$ is uniform.

\medskip


So far, we have discussed some results pertaining to finding the optimal estimator for a given TPM.
The more difficult problem is to find which TPM yields the best estimator for a given $X$ and probability distribution $\pi$ on $X$.
This problem is discussed in a general setting in 
Chan and Nayak \cite[\S 3]{CN}. They proved a lower bound which we recall now. For convenience, we assume the
TPM $Q$ is based on a BIBD.

Suppose the differential privacy ratio is denoted by $\gamma$ and
 define
\[ f(x) = \frac{v^2 (x \gamma^2 +v - x)}{(x \gamma + v - x)^2}.\]
Chan and Nayak  proved the following result in \cite[Lemma 3.2]{CN}.

\begin{theorem}
\label{CNbound.thm}
Suppose $Q$ is a TPM and $D_\rho$ is the $m\times m$ diagonal matrix with the entries of $\rho_1, \dots , \rho_m$ on the main diagonal, where $\rho$ is the probability distribution that is induced on $Y$. Then the following bound holds:
\begin{equation}
\label{CNbound.eq} \mathsf{trace}((Q^T D_{\rho}^{-1} Q)^{-1}) \geq \frac{(v-1)^2}{f(k) - v} + \frac{1}{v},
\end{equation}
\end{theorem}

Chan and Nayak  also proved in \cite[Lemma 3.3]{CN} that the lower bound (\ref{CNbound.eq}) is tight
if  $\gamma = (v-k)/k$ and  
\[ Q^T D_{\rho}^{-1} Q =  \sigma I_{v \times v} + \tau J_{v \times v},\]
where
\[ \sigma = \frac{f(k) - n}{n-1} \quad \text{and} \quad \tau = 1 - \frac{\sigma}{n}.\]
Using the formulas we have developed in Lemma \ref{QQT.eq}, this condition can be shown to hold in the scenario we are considering, i.e., when $Q$ is a TPM based on a $(v,b,r,k,\lambda)$-BIBD and the probability distribution $\pi$ on $X$ is uniform. 

It is interesting to note in the above discussion that the Chai-Nayak bound is met only when the differential privacy ratio $\gamma$ has the value $(v-k)/k$.  In general, we could start with a given BIBD, and then choose the differential privacy ratio to be $(v-k)/k$. This would ensure optimality of the TPM in the Chai-Nayak setting.
We illustrate in the following example.

\begin{example}
Suppose we take $(v,k,\lambda) = (25,4,1)$ and we set $e^{\epsilon} = (25-4)/4 = 21/4$. Then
\[ \mathsf{trace}((Q^T D_{\rho}^{-1} Q)^{-1}) = \frac{98213}{36125},\]
which meets the bound (\ref{CNbound.eq}) with equality.
\end{example}

 \section{Discussion and Summary}
 
 The construction of LDP protocols from BIBDs is an interesting application of design theory in a perhaps unexpected setting. LDP protocols can be implemented efficiently if we start with a BIBD with a rich algebraic structure. In this case, the corresponding optimal estimators will have a nice structure, since they can be constructed from the Moore-Penrose inverse of the TPM.
 
 Another aspect of LDP protocols that has been considered by some authors (e.g., \cite{PNL}) is \emph{communication cost}.
 The communication cost is defined to be $\log_2 m$, since this is the number of bits required to represent an element of $Y$. In a BIBD-based TPM, we have $m = b$ (the number of blocks in the BIBD). For given values of $v$ and $k$, the parameter $b$ is minimized by choosing $\lambda$ to be the smallest integer such that a 
 $(v,k,\lambda)$-BIBD exists.
 
 \section*{Acknowledgements}
 
 We thank the referees for very useful comments. In particular, one referee pointed out some errors in Section \ref{pure.sec} and provided Example \ref{unequal.ex} as a counterexample to an incorrectly stated previous version of Theorem \ref{main.thm}.

  \end{document}